\documentclass[12pt]{amsart}

\usepackage{amsthm,amsmath,amsfonts,amssymb}
\usepackage{mathrsfs,latexsym,enumitem,stmaryrd}
\usepackage{graphicx,caption,subcaption}
\usepackage[normalem]{ulem}
\usepackage{fourier}
\usepackage{tkz-euclide}
\usetkzobj{all}

\newtheorem{theorem}{Theorem}[section]
\newtheorem{lemma}{Lemma}[section]

\newtheorem{remark}{Remark}[section]

\numberwithin{equation}{section}

\begin{document}
\title[Spectral representation of Liouville Brownian Motion]{Spectral representation of one-dimensional Liouville Brownian Motion and Liouville Brownian excursion}

\author{Xiong Jin}
\address{School of Mathematics, University of Manchester, Oxford Road, Manchester M13 9PL, United Kingdom}
\email{xiong.jin@manchester.ac.uk}

\begin{abstract}
In this paper we apply the spectral theory of linear diffusions to study the one-dimensional Liouville Brownian Motion and Liouville Brownian excursions from a given point. As an application we estimate the fractal dimensions of level sets of one-dimensional Liouville Brownian motion as well as various probabilistic asymptotic behaviours of Liouville Brownian motion and Liouville Brownian excursions.
\end{abstract}

\maketitle

\section{Intorduction}

Liouville Brownian motion (LBM) was introduced by Garban, Rhodes and Vargas \cite{GRV16} and  by Berestycki \cite{Be15} as a way of understanding better the geometry of two-dimensional Liouville quantum gravity (LQG). Roughly speaking, planar Liouville Brownian motion is the time-change of a planar Brownian motion by the additive functional  whose Revuz measure with respect to Lebesgue is the so-called Liouville measure
\[
\mu_\gamma(\mathrm{d}z)=e^{ \gamma \mathfrak{h}(x)} \,\mathrm{d}x, \ x \in D,
\]
where $\gamma\ge 0$ is a given parameter, $D$ is a regular planar domain and $\mathfrak{h}$ is a Gaussian free field (GFF) on $D$ with certain boundary conditions. As GFFs are defined as random distributions or Gaussian processes on a certain space of measures which does not contain Dirac masses, $\mathfrak{h}(x)$ is not well-defined for individual points $x\in D$. Therefore certain smooth approximations of $\mathfrak{h}$ are needed to define the measure $\mu_\gamma$ rigorously. This was done by Duplantier and Sheffield in \cite{DS11} for $\gamma\in[0,2)$ by using circle averages around given points. The resulting measure $\mu_\gamma$ is a random measure on $D$ carried by a random fractal set whose fractal dimension is $2-\gamma^2/2$.

Such random fractal measures with ``log-Gaussian'' densities obtained via a limiting procedure have a long history. The study was initiated by Mandelbrot \cite{Ma72} in the 1970s to analyse the energy dissipation phenomenon in fully developed turbulence. The mathematically rigorous foundation of these random measures was built later by Kahane \cite{Ka85} in 1985, which now is referred as the Gaussian multiplicative chaos (GMC) theory. For a historical review of GMC and its relation to GFF and LQG, see for example the survey paper \cite{RV14a} of Rhodes and Vargas and the lecture notes \cite{Be16} of Berestycki.

The study of planar Liouville Brownian motion was carried out in \cite{GRV14,RV14a,AK16} with a focus on the regularity of the transition density function of LBM (so-called Liouville heat kernels). In this paper we shall continue the study but mainly focus on the case of one-dimensional Liouville Brownian motion, defined as a generalized linear diffusion process with natural scale function and speed measure $\nu$, where $\nu$ is a boundary Liouville measure on $\mathbb{R}$ obtained from a GFF on the upper half-plane with Neumann boundary conditions.

The advantage of studying the one-dimensional case is that there exists in the literature a fully developed theory on the probabilistic interpretation of linear diffusions in terms of their scale functions and speed measures, namely the spectral theory of linear diffusions (see \cite{DM76} for example). With the help of the spectral theory of linear diffusions, we are able to estimate various probabilistic asymptotic behaviours of one-dimensional LBM as well as that of Liouville Brownian excursions (LBE) from a given point. For example in Theorem \ref{level} we calculate the Hausdorff and packing dimension of the level sets of the one-dimensional LBM, and in Theorem \ref{ellong} and \ref{elshort} we study the asymptotic behaviours of the lifetime of the excursion under the Liouville Brownian excursion measure.

The rest of the paper is organized as follows: in Section \ref{LBM} we give a brief review of one-dimensional Brownian motion and Brownian excursions, then we define the one-dimensional LBM and LBE from a given point via time-change of additive functionals obtained from boundary Liouville measures; in Section \ref{srLBM} we give a brief review of Krein's spectral theory of strings and list the spectral representation of LBM and LBE using the spectral theory of excursions of linear diffusions developed in \cite{Ya06,SVY07}; in Section \ref{aLBM} we study various probabilistic asymptotic behaviours  of LBM and LBE.

\section{One dimensional Liouville Brownian motion and Liouville Brownian excursion}\label{LBM}

\subsection{Brownian motion and Brownian excursion}

Let $W=C(\mathbb{R}_+,\mathbb{R})$ denote the Wiener space consisting of continuous functions from $\mathbb{R}_+$ to $\mathbb{R}$. We regard $W$ as a complete separable metric space. Let $\mathcal{W}$ denote its Borel $\sigma$-field. Let $\mathbf{P}_{\mathrm{BM}}^0$ be the Wiener measure on $(W,\mathcal{W})$, under which the canonical process $w=\{w(t)\}_{t\ge 0}$ is a one-dimensional Brownian motion starting from $0$. For $x\in \mathbb{R}$ let $\mathbf{P}^x_{\mathrm{BM}}(\cdot)$ denote the measure $\mathbf{P}_{\mathrm{BM}}^0(\cdot+x)$, that is the law of the one-dimensional Brownian motion starting from $x$.

Let $\{L(t,x)\}_{t\ge 0, x\in \mathbb{R}}$ denote the joint-continuous version of the local time of the Brownian motion under $\mathbf{P}_{\mathrm{BM}}^0$. For any bounded continuous function $f$ on $\mathbb{R}$ one has
\[
\int_0^t f(w(s)) \, \mathrm{d}s =2\int_\mathbb{R} f(y) L(t,x) \, \mathrm{d}x
\]
for $\mathbf{P}_{\mathrm{BM}}^0$-almost every $w\in W$.  For $\ell\ge 0$ let
\[
\tau(\ell)=\inf\{t\ge 0: L(t,0)> \ell\}
\]
be the right-continuous inverse of the local time at $0$. For $\ell \ge 0$ such that $\tau(\ell-)<\tau(\ell)$, we may define the path of the excursion at $\ell$ as
\[
e^{(\ell)}(t)=\left\{\begin{array}{ll} |w(\tau(\ell-)+t)| & \text{if } 0\le t \le \tau(\ell)-\tau(\ell-),\\
0 & \text{if } t> \tau(\ell)-\tau(\ell-).
\end{array}\right.
\] 
The excursion $e^{(\ell)}$ takes values in the subspace $E\subset W$ consisting of continuous paths $e:[0,\infty)\mapsto [0,\infty)$ such that if $e(t_0)=0$ for some $t_0>0$ then $e(t)=0$ for all $t> t_0$. Let $\mathcal{E}$ denote its Borel $\sigma$-field. By Ito's excursion theory, there exists a $\sigma$-finite measure $\mathbf{n}_{\mathrm{BE}}$ on $(E,\mathcal{E})$ such that under $\mathbf{P}_{\mathrm{BM}}^0$, the point measure
\[
\sum_{\ell\ge 0: \tau(\ell-)<\tau(\ell)} \delta_{\ell,e^{(\ell)}} (\mathrm{d}s\mathrm{d}e)
\]
is a Poisson measure on $\mathbb{R}_+\times E$ with intensity 
\[
\mathrm{d}s \otimes \mathbf{n}_{\mathrm{BE}}(\mathrm{d}e).
\]
The measure $\mathbf{n}_{\mathrm{BE}}$ is called Ito's excursion measure of the Brownian motion. Here we present the following four descriptions of $\mathbf{n}_{\mathrm{BE}}$ listed in \cite{FY08}. For more details on Brownian excursions, see \cite[Chapter XII]{RY99} for example.

For $x>0$ let $\mathbf{Q}^x_{\mathrm{BM}}$ denote the law of the one-dimensional Brownian motion starting form $x$ and absorbed at $0$. For $x\ge 0$ let $\mathbf{P}^x_{\mathrm{3B}}$ denote the law of the $3$-dimensional Bessel process starting from $x$. Let $\mathbf{W}^x_{\mathrm{3B}}$ denote the law of the path obtained by piecing together two independent $\mathbf{P}^0_{\mathrm{3B}}$-process up to their first hitting time to $x$ (the second one runs backward in time). These measures may be all considered to be defined on $(E,\mathcal{E})$. For $e\in E$ let $M(e)=\max_{t\ge 0} e(t)$ denote the maximum of $e$ and let $\zeta(e)=\inf\{t>0:e(t)=0\}$ denote the lifetime of $e$, with the convention that $\inf \emptyset=\infty$.
\begin{itemize}
\item[(i)]  We have $\mathbf{n}_{\mathrm{BE}}(M=0)=0$ and for every bounded continuous functional $F$ on $E$ supported by $\{M>x\}$ for some $x>0$,
\[
\mathbf{n}_{\mathrm{BE}}(F)=\lim_{\epsilon\to 0+} \frac{1}{\epsilon} \mathbf{Q}^\epsilon_{\mathrm{BM}}(F).
\]
\item[(ii)] Under $\mathbf{n}_{\mathrm{BE}}$ the excursion process $e=\{e(t)\}_{t\ge 0}$ is a strong Markov process with transition kernel $\mathbf{Q}^x_{\mathrm{BM}}(e(t)\in \mathrm{d}y)$ and entrance law $\frac{1}{x}\mathbf{P}^0_{\mathrm{3B}}(e(t)\in \mathrm{d}x)$. In particular for each positive stopping time $\tau$ and every measurable set $\Gamma$,
\[
\mathbf{n}_{\mathrm{BE}}(e(\tau+\cdot)\in \Gamma)=\int_{(0,\infty)} \frac{1}{x}\mathbf{P}^0_{\mathrm{3B}}(e(\tau)\in \mathrm{d}x)\mathbf{Q}^x_{\mathrm{BM}}(\Gamma).
\]
\item[(iii)]  For every measurable set $\Gamma$,
\[
\mathbf{n}_{\mathrm{BE}}(\Gamma)=\int_0^\infty \mathbf{W}^x_{\mathrm{3B}}(\Gamma) \, \frac{\mathrm{d}x}{x^2}.
\]
This means that $\mathbf{n}_{\mathrm{BE}}(M\in \mathrm{d}x)=\frac{\mathrm{d}x}{x^2}$ and the law of $\mathbf{n}_{\mathrm{BE}}$ conditioned on $M=x$ is $\mathbf{W}^x_{\mathrm{3B}}$.
\item[(iv)] For every measurable set $\Gamma$,
\[
\mathbf{n}_{\mathrm{BE}}(\Gamma)=\int_0^\infty \mathbf{P}^0_{\mathrm{3B}}(e_t\in \Gamma \,|\, e(t)=0) p_{\mathrm{3B}}(t,0,0) \, \mathrm{d}t.
\]
Here $e_t(\cdot)=e(t\wedge \cdot)$ is the path of $e$ stopped at $t$, and $p_{\mathrm{3B}}(t,0,0)=(2\pi t)^{-\frac{3}{2}}$ is obtained from the transition probability density $p_{\mathrm{3B}}(t,x,y)$ of $\mathbf{P}^x_{\mathrm{3B}}$ with respect to its speed measure $y^2 \mathrm{d}y$ evaluated at $x=y=0$.   Since $ \mathbf{P}^0_{\mathrm{3B}}(\zeta(e)=t \,|\, e(t)=0)=1$, this description means that $\mathbf{n}_{\mathrm{BE}}(\xi\in \mathrm{d}t)= p_{\mathrm{3B}}(t,0,0)\mathrm{d}t$ and the law of $\mathbf{n}_{\mathrm{BE}}$ conditioned on $\zeta=t$ is $\mathbf{P}^0_{\mathrm{3B}}(e_t\in\cdot \,|\, e(t)=0)$.
\end{itemize}

\subsection{Gaussian multiplicative chaos and Liouville quantum gravity}\label{lm}

Let $D\subset \mathbb{R}^k$ be a domain. Let $K(x,y)$ be a nonnegative definite kernel of the form
\[
-\log|x-y|+g(x,y),
\]
where $g$ is continuous over $\overline{D}\times \overline{D}$. Let
\[
\mathfrak{M}_+=\left\{\sigma\text{-finite measure } \rho \text{ on $D$ with } \int_D\int_D K(x,y) \, \rho(\mathrm{d}x)\rho(\mathrm{d}y)<\infty\right\}
\]
and let $\mathfrak{M}$ be the set of the signed measures of the form $\rho=\rho_+-\rho_-$, where $\rho_+,\rho_-\in\mathfrak{M}_+$. Let $\mathfrak{h}=\{\mathfrak{h}(\rho)\}_{\rho\in\mathfrak{M}}$ be a centered Gaussian process indexed by $\mathfrak{M}$ with covariance function
\[
\mathrm{Cov}(\mathfrak{h}(\rho),\mathfrak{h}(\rho'))=\int_D\int_D K(x,y) \, \rho(\mathrm{d}x)\rho'(\mathrm{d}y).
\]
The process $\mathfrak{h}$ is called a Gaussian field on $D$ with covariance kernel $K$. Let $\theta$ be a smooth mollifier and for $\epsilon>0$ let $\theta_\epsilon(x)=\epsilon^{-k}\theta(x/\epsilon)$. Let $\mathfrak{h}_\epsilon(x)=\mathfrak{h}*\theta_\epsilon(x)$ be the smooth approximation of $\mathfrak{h}$. Then for $\epsilon>0$ we may define a random measure
\[
\mu_{\gamma,\epsilon}(\mathrm{d}x)=e^{\gamma h_\epsilon(x)-\frac{\gamma^2}{2}\mathbb{E}(\mathfrak{h}_\epsilon(x))} \, \mathrm{d}x,\  x\in D,
\]
where $\gamma\ge 0$ is a given parameter. When $\gamma<\sqrt{2k}$, the sequence of measures $\mu_{\gamma,\epsilon}$ converges weakly in probability to a limiting measure $\mu_\gamma $ called a Gaussian multiplicative chaos measure on $D$ (See \cite{Be15a} for example for an elementary proof).

Here we shall focus on the case when $D=\mathbb{H}$ is the upper half-plane and consider the boundary Liouville measures on $\mathbb{R}$ defined as follows: let $g(x,y)=-\log|x-\bar{y}|$, then $\mathfrak{h}=\mathfrak{h}^{\mathrm{f}}$ is the Gaussian free field on $\mathbb{H}$ with Neumann boundary conditions. For $x\in \mathbb{R}$ and $\epsilon>0$ let $\rho_{x,\epsilon}$ denote the Lebesgue measure on the semi-circle $\{y\in \mathbb{H}: |y-x|=\epsilon\}$ in $\mathbb{H}$ normalized to have mass $1$. Let $\gamma \in[0,\sqrt{2})$ be fixed. For $n\ge 1$ define
\[
\nu_{n}(\mathrm{d}x)= 2^{-n\frac{\gamma^2}{2}} e^{\frac{\gamma}{\sqrt{2}} \mathfrak{h}^\mathrm{f}(\rho_{x,2^{-n}})} \, \mathrm{d}x, \ x\in \mathbb{R}.
\]
Then almost surely $\nu_{n}$ converge weakly to a non-trivial measure $\nu$ as $n\to \infty$. The measure $\nu$ is called the boundary Liouville measure on $\mathbb{R}$ with parameter $\gamma$.

\subsection{One-dimensional Liouville Brownian motion}

We assume that the GFF $\mathfrak{h}$ and the Brownian motion are independent of each other. Let $\nu$ be an instance of the boundary Liouville measure on $\mathbb{R}$ with parameter $\gamma$ as constructed in Section \ref{lm}. Define
\[
A_{\nu}(t)=\int_\mathbb{R} L(t,x) \,\nu(\mathrm{d}x), t\ge 0.
\]
Then $A_{\nu}=\{A_\nu(t)\}_{t\ge 0}$ forms an additive functional of the Brownian motion. Let
\[
\tau_{\nu}(t)=\inf\{s\ge 0: A_{\nu}(s)>t\}, t\ge 0
\]
be its right-continuous inverse. For $w\in W$ let
\[
w_{\nu}(t)=w(\tau_{\nu}(t)), t\ge 0
\]
denote the time-change of $w$ by $\tau_\nu$. Define a probability measure on $W$ by
\[
\mathbf{P}^x_{\nu}(\cdot)=\mathbf{P}^x_{\mathrm{BM}}(w_{\nu}\in \cdot).
\]
Then $\mathbf{P}^x_{\nu}$ is the law of the one-dimensional Liouville Brownian motion with respect to the boundary Liouville measure $\nu$. In other words, one-dimensional Liouville Brownian motion is a generalized linear diffusion process on $\mathbb{R}$ with natural scale function and speed measure $\nu$. Its joint-continuous transition density $p_{\nu}(t;x,y)$ is given by
\[
\mathbf{P}^x_{\nu}(w(t)\in B)=\int_B p_{\nu}(t;x,y) \, \nu(\mathrm{d}y)
\]
for $t>0$, $x\in\mathbb{R}$ and $B\in\mathcal{B}(\mathbb{R})$. The process $\{L_\nu(t,x)=L(\tau_\nu(t),x)\}_{t\ge 0, x\in \mathbb{R}}$ is the joint-continuous local time of the Liouvile Brownian motion under $\mathbf{P}^0_{\nu}$, that is, for any bounded continuous function $f$ on $\mathbb{R}$ one has
\[
\int_0^t f(w(s)) \, \mathrm{d}s =2\int_\mathbb{R} f(y) L_\nu(t,x) \, \nu(\mathrm{d}x)
\]
for $\mathbf{P}_{\nu}^0$-almost every $w\in W$.

\subsection{Liouville Brownian excursion}

Fix $a\in\mathbb{R}$. For $x\ge 0$ denote by
\[
m_{\nu,a,+}(x)=\nu([a,a+x]) \text{ and } m_{\nu,a,-}(x)=\nu([a-x,a]).
\]
Define
\[
A_{\nu,a,\pm}(t)=\int_{(0,\infty)} L(t,a\pm x) \, \mathrm{d}m_{\nu,a,\pm}(x), \ t\ge 0.
\]
Then $A_{\nu,a,\pm}$ forms an additive functional of the Brownian motion. Let
\[
\tau_{\nu,a,\pm}(t)=\inf\{s\ge 0: A_{\nu,a,\pm}(s)>t\}, t\ge 0
\]
be the right-continuous inverse of $A_{\nu,a,\pm}$. For $w\in W$ let
\[
w_{\nu,a,\pm}(t)=w(\tau_{\nu,a,\pm}(t)), \ t\ge 0
\]
be the time change of $w$ by $\tau_{\nu,a,\pm}$. For $x>0$ define a probability measure on $W$ by
\[
\mathbf{P}^{x}_{\nu,a,\pm}(\cdot)=\mathbf{P}^{a\pm x}_{\mathrm{BM}}(\pm(w_{\nu,a,\pm}-a)\in \cdot).
\]
We shall also use the same notation $L(t,x)$ to denote the joint-continuous version of the local time of the Brownian motion/excursion under $\mathbf{Q}^x_{\mathrm{BM}}$, $\mathbf{P}^x_{\mathrm{3B}}$, $\mathbf{W}^x_{\mathrm{3B}}$ and $\mathbf{n}_{\mathrm{BE}}$ on $E$. For $e\in E$ let
\[
A_{\nu,a,\pm}(t)=\int_{(0,\infty)} L(t, x) \, \mathrm{d}m_{\nu,a,\pm}(x), \ t\ge 0,
\]
and
\[
\tau_{\nu,a,\pm}(t)=\left\{\begin{array}{ll} \inf\{s\ge 0: A_{\nu,a,\pm}(s)>t\} & \text{if } 0\le t<A_{\nu,a,\pm}(\zeta);\\
\zeta & \text{if } t\ge A_{\nu,a,\pm}(\zeta),
\end{array}\right.
\]
as well as
\[
e_{\nu,a,\pm}(t)=e(\tau_{\nu,a,\pm}(t)), \ t\ge 0.
\]
Define the measures on $E$ by
\begin{align*}
\mathbf{Q}_{\nu,a,\pm}^{x}(\cdot)=&\mathbf{Q}^x_{\mathrm{BM}}(e_{\nu,a,\pm}\in \cdot),\\
\mathbf{Q}^{x}_{h\text{-}\nu,a,\pm}(\cdot)=&\mathbf{P}^x_{\mathrm{3B}}(e_{\nu,a,\pm}\in \cdot), \\
\mathbf{W}_{\nu,a,\pm}^{x}(\cdot)=&\mathbf{W}^x_{\mathrm{BM}}(e_{\nu,a,\pm}\in \cdot),\\
\mathbf{n}_{\nu,a,\pm}(\cdot)=&\mathbf{n}_{\mathrm{BE}}(e_{\nu,a,\pm}\in\cdot).
\end{align*}
We have
\begin{itemize}
\item[(1)] $\mathbf{P}^{x}_{\nu,a,\pm}$ is the law of the generalized linear diffusion with natural scale function and speed measure $\mathrm{d}m_{\nu,a,\pm}(x)$ on $\mathbb{R}_+$ starting from $x$ and with $0$ as an instantaneously reflecting boundary.
\item[(2)] $\mathbf{Q}_{\nu,a,\pm}^{x}$ is the law of the generalized linear diffusion with natural scale function and speed measure $\mathrm{d}m_{\nu,a,\pm}(x)$ on $\mathbb{R}_+$ starting from $x$ and absorbed at $0$.
\item[(3)] $\mathbf{Q}^{x}_{h\text{-}\nu,a,\pm}$ is the law of  the generalized linear diffusion with natural scale function and speed measure $\mathrm{d}m_{\nu,a,\pm}(x)$ on $\mathbb{R}_+$ starting from $x$ and conditioned never hit $0$. Indeed $\mathbf{Q}^{x}_{h\text{-}\nu,a,\pm}$ is Doob's $h$-transform of $\mathbf{Q}_{\nu,a,\pm}^{x}$ with $h(x)=x$. It is therefore the law of the generalized linear diffusion with speed measure $x^2 \mathrm{d}m_{\nu,a,\pm}(x)$ and scale function $-1/x$.
\item[(4)] $\mathbf{W}_{\nu,a,\pm}^{x}$ is the law of the following process: consider two independent $\mathbf{Q}^{0}_{h\text{-}\nu,a,\pm}$-processes until they first hit $x$ and splice the two paths together (the second one runs backward in time).
\end{itemize}

Finally, by applying \cite[Theorem 2.5]{FY08}, we have the following descriptions of the Ito's excursion measure $\mathbf{n}_{\nu,a,\pm}$.

\begin{theorem}\ 
\begin{itemize}
\item[(i)] We have $\mathbf{n}_{\nu,a,\pm}(M=0)=0$ and for every bounded continuous functional $F$ on $E$ supported by $\{M>x\}$ for some $x>0$,
\[
\mathbf{n}_{\nu,a,\pm}(F)=\lim_{\epsilon\to 0+} \frac{1}{\epsilon} \mathbf{Q}_{\nu,a,\pm}^{\epsilon}(F).
\]
\item[(ii)] Under $\mathbf{n}_{\nu,a,\pm}$ the excursion process $\{e(t)\}_{t\ge 0}$ is a strong Markov process with the transition kernel $\mathbf{Q}_{\nu,a,\pm}^{x}(e(t)\in \mathrm{d}y)$ and the entrance law $\frac{1}{x}\mathbf{Q}^{0}_{h\text{-}\nu,a,\pm}(e(t)\in \mathrm{d}x)$. In particular for each positive stopping time $\tau$ and every measurable set $\Gamma$,
\[
\mathbf{n}_{\nu,a,\pm}(e(\tau+\cdot)\in \Gamma)=\int_{(0,\infty)} \frac{1}{x}\mathbf{Q}^{0}_{h\text{-}\nu,a,\pm}(e(\tau)\in \mathrm{d}x)\, \mathbf{Q}_{\nu,a,\pm}^{x}(\Gamma).
\]
\item[(iii)]  For every measurable set $\Gamma$,
\[
\mathbf{n}_{\nu,a,\pm}(\Gamma)=\int_0^\infty \mathbf{W}_{\nu,a,\pm}^{x}(\Gamma) \, \frac{\mathrm{d}x}{x^2}.
\]
This means that $\mathbf{n}_{\nu,a,\pm}(M\in \mathrm{d}x)=\frac{\mathrm{d}x}{x^2}$ and the law of $\mathbf{n}_{\nu,a,\pm}$ conditioned on $M=x$ is $\mathbf{W}_{\nu,a,\pm}^{x}$.
\end{itemize}
\end{theorem}

\section{Spectral representation of Liouville Brownian motion and Liouville Brownian excursion}\label{srLBM}

\subsection{Krein's spectral theory of strings}\label{kr}

This section is based on \cite{KW82}. Let $\mathcal{M}$ be the set of non-decreasing right-continuous functions $m:[0,\infty]\mapsto [0,\infty]$ with $m(0-)=0$ and $m(\infty)=\infty$. Each $m\in \mathcal{M}$ represents the mass distribution of a string. For $m\in \mathcal{M}$ let $l=\sup\{x\ge 0: m(x)<\infty\}$ denote the length of  $m$. For $\lambda\in\mathbb{C}$ let $\varphi(x,\lambda)$ and $\psi(x,\lambda)$ be the unique solution of the following integral equations on $[0,l)$ respectively:
\begin{eqnarray*}
\varphi(x,\lambda) &=& 1+\lambda\int_{(0,x]} (x-y) \varphi(y,\lambda) \, \mathrm{d}m(y);\\
\psi(x,\lambda) &=& x+\lambda\int_{(0,x]} (x-y) \psi(y,\lambda) \, \mathrm{d} m(y).
\end{eqnarray*}
The functions $\varphi$ and $\psi$ have the following explicit expressions: let
\[
\varphi_0(x)=1; \ \varphi_{n +1}(x)=\int_{(0,x]}(x-y) \varphi_n(y) \, \mathrm{d}m(y) \text{ for } n\ge 0,
\]
\[
\psi_0(x)=x; \ \psi_{n +1}(x)=\int_{(0,x]}(x-y) \psi_n(y) \, \mathrm{d}m(y) \text{ for } n\ge 0,
\]
then
\[
\varphi(x,\lambda)=\sum_{n=0}^\infty \varphi_n(x) \lambda^n;\ 
\psi(x,\lambda)=\sum_{n=0}^\infty \psi_n(x) \lambda^n.
\]
For each fixed $x\in [0,l)$, $\phi(x,\cdot)$ and $\psi(x,\cdot)$ are real entire functions, i.e., they are entire functions of $\lambda$ and they take real values if $\lambda\in\mathbb{R}$. Set
\[
h(\lambda)=\int_0^l \frac{\mathrm{d} x}{\varphi(x,\lambda)^2}=\lim_{x\uparrow l} \frac{\psi(x,\lambda)}{\varphi(x,\lambda)}.
\]
The function $h$ is called Krein's correspondence of the string $m$.

Let $\mathcal{H}$ be the set of functions $h:(0,\infty)\mapsto \mathbb{C}$ such that $h(\lambda)$ can be extended to a homomorphic function on $\mathbb{C}\setminus (-\infty,0]$ such that $\mathrm{Im} h(\lambda) \le 0$ for $\lambda\in \mathbb{C}$ with $\mathrm{Im} \lambda>0$ and $h(\lambda)>0$ for $\lambda>0$. Introduce the topology on $\mathcal{M}$ such that $m_n\mapsto m$ if and only if $m_n(x)\mapsto m(x)$ on every continuous point of $m$, and the topology on $\mathcal{H}$ such that $h_n\mapsto h$ if and only if $h_n(\lambda)\mapsto h(\lambda)$ for every $\lambda>0$.

\begin{theorem}[Krein's correspondence]
$\mathcal{M}$ and $\mathcal{H}$ are compact metric spaces and Krein's correspondence $m\in\mathcal{M}\leftrightarrow h\in\mathcal{H}$ defines a homeomorphism. Moreover, $h\in \mathcal{H}$ has a unique representation
\[
h(\lambda)=c+\int_{[0,\infty)} \frac{\sigma(\mathrm{d}\xi)}{\lambda+\xi},
\]
where $c=\inf\{x>0: m(x)>0\}$ and $\sigma$ is a non-negative Borel measure on $[0,\infty)$ with $\int_{(0,\infty)} \frac{\sigma(\mathrm{d}\xi)}{1+\xi}<\infty$.
\end{theorem}

The unique Borel measure $\sigma$ is called the spectral measure of $m$. From the functional analysis point of view, $\sigma$ is the unique measure on $[0,\infty)$ such that for $f\in L^2([0,l), \mathrm{d}m)$,
\[
\|f\|_{L^2([0,l), \mathrm{d}m)}=\|\hat{f}\|_{L^2([0,\infty), \sigma)},
\]
where
\[
\hat{f}(\lambda)=\int_0^l f(x)\varphi(x,\lambda) \, \mathrm{d}m(x)
\]
is the generalized Fourier transform.

Note that the right-continuous inverse $m^*:t\ge 0 \mapsto \inf\{x>0: m(x)>t\}$ also belongs to $\mathcal{M}$ with length $l^*=m(\infty-)$. It is called the dual string of $m$. Its Krein's correspondence is given by $h^*(\lambda)=\frac{1}{\lambda h(\lambda)}$, which also has a unique representation
\[
h^*(\lambda)=c^*+\int_{[0,\infty)} \frac{\sigma^*(\mathrm{d}\xi)}{\lambda+\xi},
\]
where $c^*=m(0)$ and $\sigma^*$ is a non-negative Borel measure supported on $[0,\infty)$ with $\int_{(0,\infty)} \frac{\sigma^*(\mathrm{d}\xi)}{1+\xi}<\infty$. The measure $\sigma^*$ is called the spectral measure of the dual string of $m$.

\subsection{Spectral representation of Liouville Brownian motion}\label{srLBM}

Throughout this Section and Section \ref{srLBE} let $\nu$ be a Borel measure on $\mathbb{R}$ satisfying
\begin{itemize}
\item[(A1)] $\nu$ has no atoms; $0<\nu([a,b])<\infty$ for all $-\infty<a<b<\infty$; $\nu([0,x]) \to \infty$ and $\nu([-x,0])\to \infty$ as $x\to \infty$.
\end{itemize}
In particular an instance of the boundary Liouville measure constructed in Section \ref{lm} satisfies (A1).

Let $\mathcal{M}_c \subset\mathcal{M}$ be the set of functions $m:[0,\infty]\mapsto [0,\infty]$ that are continuous, strictly increasing functions with $m(0)=0$ and $\sup\{x:m(x)<\infty\}=\infty$. For $x\ge 0$ define $m_{\nu,+}(x)=\nu([0,0+x])$ and $m_{\nu,-}(x)=\nu([-x,0])$. Then by (A1) both $m_{\nu,+}$ and $m_{\nu,-}$ belong to $\mathcal{M}_c$. Let $\varphi_{\nu,\pm}(x,\lambda)$ and $\psi_{\nu,\pm}(x,\lambda)$ be the unique solutions of the integral equations
\begin{eqnarray*}
\varphi_{\nu,\pm}(x,\lambda) &=& 1+\lambda\int_{(0,x]} (x-y) \varphi_{\nu,\pm}(y,\lambda) \, \mathrm{d}m_{\nu,\pm}(y);\\
\psi_{\nu,\pm}(x,\lambda) &=& x+\lambda\int_{(0,x]} (x-y) \psi_{\nu,\pm}(y,\lambda) \, \mathrm{d} m_{\nu,\pm}(y),
\end{eqnarray*}
and let
\[
h_{\nu,\pm}(\lambda)=\int_0^\infty \frac{\mathrm{d}x}{\varphi_{\nu,\pm}(x,\lambda)}=\lim_{x\to \infty} \frac{\psi_{\nu,\pm}(x,\lambda)}{\varphi_{\nu,\pm}(x,\lambda)}
\]
be the Krein's correspondence of $m_{\nu,\pm}$. Let $\sigma_{\nu,\pm}$ be the spectral measure of $m_{\nu,\pm}$, that is the unique non-negative Borel measure on $[0,\infty)$ with $\int_{(0,\infty)} \frac{\sigma_{\nu,\pm}(\mathrm{d}\xi)}{1+\xi}<\infty$ such that
\[
h_{\nu,\pm}(\lambda)=\int_0^\infty \frac{\sigma_{\nu,\pm}(\mathrm{d}\xi)}{\lambda+\xi}.
\]
Let $h_{\nu}$ be the Krein's correspondence of $m_{\nu}=m_{\nu,+}+m_{\nu,-}$, which satisfies
\[
\frac{1}{h_{\nu}(\lambda)}=\frac{1}{h_{\nu,+}(\lambda)}+\frac{1}{h_{\nu,-}(\lambda)}.
\]
Let $\sigma_{\nu}$ be the spectral measure of $m_{\nu}$, that is the unique non-negative Borel measure on $[0,\infty)$ with $\int_{(0,\infty)} \frac{\sigma_{\nu}(\mathrm{d}\xi)}{1+\xi}<\infty$ such that
\[
h_{\nu}(\lambda)=\int_0^\infty \frac{\sigma_{\nu}(\mathrm{d}\xi)}{\lambda+\xi}.
\]
Define
\begin{align*}
\varphi_{\nu}(x,\lambda)=&\left\{\begin{array}{ll}\varphi_{\nu,+}(x,\lambda) & \text{if } x\ge 0;\\ \varphi_{\nu,-}(-x,\lambda) & \text{if } x<0; \end{array}\right.\\
\psi_{\nu}(x,\lambda)=&\left\{\begin{array}{ll}\psi_{\nu,+}(x,\lambda) & \text{if } x\ge 0;\\ -\psi_{\nu,-}(-x,\lambda) & \text{if } x<0. \end{array}\right.
\end{align*}
For $\lambda>0$ the $\lambda$-resolvent operator $G^\lambda_\nu$ of the LBM is defined as
\[
G_\nu^\lambda f(x)=\mathbf{E}^x_\nu\left(\int_0^\infty e^{-\lambda t} f(w(t)) \, \mathrm{d} t\right)
\]
for any bounded continuous function $f$ on $\mathbb{R}$. We have the following spectral representation:
\[
G_\nu^\lambda f(x)=\int_\mathbb{R} g_\nu^\lambda(x,y) f(y) \, \nu(\mathrm{d} y),
\]
where the $\lambda$-resolvent kernel $g_\nu^\lambda(x,y)$ is given by
\[
g_\nu^\lambda(x,y)=h_\nu(\lambda)(\varphi_\nu(x,\lambda)+h_{\nu,+}(\lambda)^{-1}\psi_\nu(x,\lambda))(\varphi_\nu(y,\lambda)-h_{\nu,-}(\lambda)^{-1}\psi_\nu(y,\lambda)).
\]
Many probabilistic quantities of the LBM are related to the $\lambda$-resolvent kernel $g_\nu^\lambda(x,y)$. For example
\begin{itemize}
\item[(i)] It is the Laplace transform of the transition density $p_\nu(t;x,y)$:
\[
g_\nu^\lambda(x,y)=\int_0^\infty e^{-\lambda t} p_\nu(t;x,y) \, \mathrm{d}t.
\]
\item[(ii)] The right-continuous inverse of the local time $\ell_{\nu,0}(t)=\inf\{s\ge 0: L_\nu(s,0)>t\}$ at $0$ is a L\'evy subordinator, whose L\'evy exponent is given by
\[
\mathbf{E}_\nu^0\left(e^{-\lambda \ell_{\nu,0}(t)}\right)=e^{-t/g_\nu^\lambda(0,0)}=e^{-t/h_\nu(\lambda)}.
\]
\item[(iii)] For $a\in \mathbb{R}$ let $H_a=\inf\{t>0: w(t)=a\}$ denote the first hitting time at $a$. Then for $a,b\in \mathbb{R}$ we have
\[
\mathbf{E}^a_\nu(e^{-\lambda H_b})=\frac{g_\nu^\lambda(a,b)}{g_\nu^\lambda(b,b)}.
\]
\end{itemize}

\subsection{Spectral representation of Liouville Brownian excursion}\label{srLBE}
Fix $a\in\mathbb{R}$. Recall the strings $m_{\nu,a,+}(x)=\nu([a,a+x])$ and $m_{\nu,a,-}(x)=\nu([a-x,a])$ for $x\ge 0$. Note that both $m_{\nu,a,+}$ and $m_{\nu,a,-}$ belong to $\mathcal{M}_c$. Let $\varphi_{\nu,a,\pm}(x,\lambda)$ and $\psi_{\nu,a,\pm}(x,\lambda)$ be the unique solutions of the integral equations
\begin{eqnarray*}
\varphi_{\nu,a,\pm}(x,\lambda) &=& 1+\lambda\int_{(0,x]} (x-y) \varphi_{\nu,a,\pm}(y,\lambda) \, \mathrm{d}m_{\nu,a,\pm}(y);\\
\psi_{\nu,a,\pm}(x,\lambda) &=& x+\lambda\int_{(0,x]} (x-y) \psi_{\nu,a,\pm}(y,\lambda) \, \mathrm{d} m_{\nu,a,\pm}(y),
\end{eqnarray*}
and let
\[
h_{\nu,a,\pm}(\lambda)=\int_0^\infty \frac{\mathrm{d}x}{\varphi_{\nu,a,\pm}(x,\lambda)}=\lim_{x\to \infty} \frac{\psi_{\nu,a,\pm}(x,\lambda)}{\varphi_{\nu,a,\pm}(x,\lambda)}
\]
be the Krein's correspondence of $m_{\nu,a,\pm}$. Let $\sigma_{\nu,a,\pm}$ be the spectral measure of $m_{\nu,a,\pm}$.  Let $m_{\nu,a,\pm}^*:t\ge 0\mapsto \inf\{s\ge 0: m_{\nu,a,\pm}(s)>t\}$ denote the dual string of $m_{\nu,a,\pm}$ and let $h_{\nu,a,\pm}^*$ denote its Krein's correspondence. We have
\[
h_{\nu,a,\pm}^*(\lambda)=\frac{1}{\lambda h_{\nu,a,\pm}(\lambda)}.
\]
Let $\sigma_{\nu,a,\pm}^*$ be the spectral measure of $m_{\nu,a,\pm}^*$. We have the following spectral representations of LBM with different boundary conditions:
\begin{itemize}
\item[(1)] Let $p_{\nu,a,\pm}(t;x,y)$ be the joint-continuous transition density of the generalized linear diffusion with natural scale function and speed measure $\mathrm{d}m_{\nu,a,\pm}(x)$ on $\mathbb{R}_+$, and with $0$ as an instantaneously reflecting boundary, that is, for $t>0$, $x,y\in [0,\infty)$ and $B\in\mathcal{B}(\mathbb{R}_+)$
\[
\mathbf{P}_{\nu,a,\pm}^{x}(w(t)\in B)=\int_B p_{\nu,a,\pm}(t;x,y) \, \mathrm{d}m_{\nu,a,\pm}(y).
\]
Then
\[
p_{\nu,a,\pm}(t;x,y)=\int_{(0,\infty)} e^{-t\lambda} \varphi_{\nu,a,\pm}(x,-\lambda)\varphi_{\nu,a,\pm}(y,-\lambda) \, \sigma_{\nu,a,\pm}(\mathrm{d}\lambda).
\]
The associated $\lambda$-resolvent kernel $g_{\nu,a,\pm}^\lambda(x,y)$ is given by
\begin{align*}
g_{\nu,a,\pm}^\lambda(x,y)=&\int_{(0,\infty)} e^{-t\lambda}p_{\nu,a,\pm}(t;x,y) \,\mathrm{d}t\\
=&\int_{(0,\infty)} \frac{\varphi_{\nu,a,\pm}(x,-\xi)\varphi_{\nu,a,\pm}(y,-\xi)}{\lambda+\xi} \, \sigma_{\nu,a,\pm}(\mathrm{d}\xi).
\end{align*}

\item[(2)] Let $q_{\nu,a,\pm}(t;x,y)$ be the joint-continuous transition density of the generalized linear diffusion with natural scale function and speed measure $\mathrm{d}m_{\nu,a,\pm}(x)$ on $\mathbb{R}_+$, and with $0$ as an absorbing boundary, that is, for $t>0$, $x,y\in(0,\infty)$ and $B\in\mathcal{B}(\mathbb{R}_+)$,
\[
\mathbf{Q}_{\nu,a,\pm}^{x}(e(t)\in B)=\int_B q_{\nu,a,\pm}(t;x,y) \, \mathrm{d}m_{\nu,a,\pm}(y).
\]
Then 
\[
q_{\nu,a,\pm}(t;x,y)=\int_{(0,\infty)} e^{-t\lambda} \psi_{\nu,a,\pm}(x,-\lambda)\psi_{\nu,a,\pm}(y,-\lambda) \, \lambda\sigma_{\nu,a,\pm}^*(\mathrm{d}\lambda).
\]
The associated $\lambda$-resolvent kernel $\hat{g}_{\nu,a,\pm}^\lambda(x,y)$ is given by
\begin{align*}
\hat{g}_{\nu,a,\pm}^\lambda(x,y)=&\int_{(0,\infty)} e^{-t\lambda}q_{\nu,a,\pm}(t;x,y) \,\mathrm{d}t\\
=&\int_{(0,\infty)} \frac{\psi_{\nu,a,\pm}(x,-\xi)\psi_{\nu,a,\pm}(y,-\xi)}{\lambda+\xi} \, \xi\sigma^*_{\nu,a,\pm}(\mathrm{d}\xi).
\end{align*}

\item[(3)] Let $q_{h\text{-}\nu,a,\pm}(t;x,y)$ be the joint-continuous transition density of the generalized linear diffusion with natural scale function and speed measure $\mathrm{d}m_{\nu,a,\pm}(x)$ on $\mathbb{R}_+$, and conditioned never hit $0$, that is,  for $t>0$, $x,y\in(0,\infty)$ and $B\in\mathcal{B}(\mathbb{R}_+)$,
\[
\mathbf{Q}^{x}_{h\text{-}\nu,a,\pm}(e(t)\in B)=\int_B q_{h\text{-}\nu,a,\pm}(t;x,y) \, \mathrm{d}m_{\nu,a,\pm}(y).
\]
Then
\begin{align*}
q_{h\text{-}\nu,a,\pm}(t;x,y)=&\frac{q_{\nu,a,\pm}(t;x,y)}{xy}\\
=&\int_{(0,\infty)} e^{-t\lambda} \frac{\psi_{\nu,a,\pm}(x,-\lambda)}{x}\frac{\psi_{\nu,a,\pm}(y,-\lambda)}{y} \, \lambda\sigma_{\nu,a,\pm}^*(\mathrm{d}\lambda).
\end{align*}
The associated $\lambda$-resolvent kernel $\widetilde{g}_{\nu,a,\pm}^\lambda(x,y)$ is given by
\begin{align*}
\widetilde{g}_{\nu,a,\pm}^\lambda(x,y)=&\int_{(0,\infty)} e^{-t\lambda}q_{h\text{-}\nu,a,\pm}(t;x,y) \,\mathrm{d}t\\
=&\int_{(0,\infty)} \frac{\psi_{\nu,a,\pm}(x,-\xi)\psi_{\nu,a,\pm}(y,-\xi)}{xy(\lambda+\xi)} \, \xi\sigma_{\nu,a,\pm}^*(\mathrm{d}\xi).
\end{align*}

\item[(4)] The partial derivative of $q_{\nu,a,\pm}(t;x,y)$ at $y=0$,
\[
\pi_{\nu,a,\pm}(t;x)=\lim_{y\to 0+} \frac{q_{\nu,a,\pm}(t;,x,y)}{y}=\int_{(0,\infty)} e^{-t\lambda} \psi_{\nu,a,\pm}(x,-\lambda) \,  \lambda\sigma_{\nu,a,\pm}^*(\mathrm{d}\lambda),
\]
is the density of the first hitting time $H_0=\inf\{t> 0:w(t)=0\}$ under $\mathbf{Q}_{\nu,a,\pm}^{x}$, that is
\[
\mathbf{Q}_{\nu,a,\pm}^{x}(H_0\in \mathrm{d}t)=\pi_{\nu,a,\pm}(t;x) \, \mathrm{d}t.
\]
In particular
\[
\mathbf{Q}_{\nu,a,\pm}^{x}(H_0>t)=\int_0^\infty e^{-t\lambda} \pi_{\nu,a,\pm}(t;x) \, \sigma_{\nu,a,\pm}^*(\mathrm{d}\lambda).
\]
It also defines an entrance law: for $t,s>0$ and $y\in(0,\infty)$,
\[
\int_{(0,\infty)} \pi_{\nu,a,\pm}(t;x) q_{\nu,a,\pm}(s;x,y) \, \mathrm{d}m_{\nu,a,\pm}(x)=\pi_{\nu,a,\pm}(t+s;y).
\]

\item[(5)] For $t>0$ let $G_t=\sup\{s\le t: w(s)=0\}$ and $D_t=\inf\{s\ge t: w(s)=0\}$. Then for $u<t<v$ and $x>0$,
\begin{align*}
\mathbf{P}_{\nu,a,\pm}^{0}(G_t&\in \mathrm{d}u, w(t)\in \mathrm{d} x, D_t\in\mathrm{d}v)\\
&=p_{\nu,a,\pm}(u;0,0)\pi_{\nu,a,\pm}(t-u;x)\pi_{\nu,a,\pm}(v-t;x)\, \mathrm{d}u\mathrm{d}v\mathrm{d}m_{\nu,a,\pm}(x).
\end{align*}

\item[(6)] The partial derivative of $\pi_{\nu,a,\pm}(t;x)$ at $x=0$,
\[
n_{\nu,a,\pm}(t)=\lim_{x\to 0+} \frac{ \pi_{\nu,a,\pm}(t;x)}{x}=\int_{(0,\infty)} e^{-t\lambda} \,  \lambda\sigma_{\nu,a,\pm}^*(\mathrm{d}\lambda),
\]
is the density of the L\'evy measure of the L\'evy subordinator $\{\ell_{\nu,a,\pm}(t):=\inf\{s\ge 0: L(\tau_{\nu,a,\pm}(s),0)>t\}\}_{t\ge 0}$ under $\mathbf{P}_{\nu,a,\pm}^{0}$, that is
\[
\mathbf{E}_{\nu,a,\pm}^{0}(e^{-\lambda \ell_{\nu,a,\pm}(t)})=e^{-t /h_{\nu,a,\pm}(\lambda)},
\]
where the L\'evy exponent $1/h_{\nu,a,\pm}(\lambda)$ takes the form
\[
\frac{1}{h_{\nu,a,\pm}(\lambda)}=\int_0^\infty (1-e^{-t\lambda}) n_{\nu,a,\pm}(t) \,\mathrm{d}t.
\]
\end{itemize}

Let $\mathbf{Q}^{0,t,0}_{\nu,a,\pm}$ denote the law of the $\mathbf{Q}^{0}_{h\text{-}\nu,a,\pm}$-process pinned at $0$ with lifetime $t$. Alternatively $\mathbf{Q}^{0,t,0}_{\nu,a,\pm}$ is the weak limit of the law of the Markovian bridge $\mathbf{Q}^{x,t,y}_{\nu,a,\pm}$ as $y\to 0+$, $x\to 0+$ (see \cite{FPY92} for example). By applying the results in \cite{Ya06,SVY07}, we have the following spectral representation of the Ito's excursion measure $\mathbf{n}_{\nu,a,\pm}$.

\begin{theorem}
The Ito's excursion measure $\mathbf{n}_{\nu,a,\pm}$ has the following representation:
\[
\mathbf{n}_{\nu,a,\pm}(\mathrm{d}e)=\int_0^\infty \mathbf{n}_{\nu,a,\pm}(\zeta\in \mathrm{d}t) \, \mathbf{Q}^{0,t,0}_{\nu,a,\pm}(\mathrm{d}e),
\]
where the law of the lifetime $\zeta$ under $\mathbf{n}_{\nu,a,\pm}$ is equal to the L\'evy measure of $\ell_{\nu,a,\pm}$:
\[
\mathbf{n}_{\nu,a,\pm}(\zeta\in dt)=n_{\nu,a,\pm}(t)\, \mathrm{d}t.
\]
Moreover, we have the following finite dimensional distribution: for $0<t_1<t_2<\cdots<t_n$ and $x_i>0$, $i=1,\ldots,n$,
\begin{align*}
\mathbf{n}_{\nu,a,\pm}(e(t_1)\in \mathrm{d}x_1,e(t_2)\in \mathrm{d}x_2,\ldots,e(t_n)\in\mathrm{d}x_n)&\\
=\pi_{\nu,a,\pm}(t_1;x_1)\mathrm{d}m_{\nu,a,\pm}(x_1) q_{\nu,a,\pm}&(t_2-t_1;x_1,x_2) \mathrm{d}m_{\nu,a,\pm}(x_2)\\
\times\cdots \times q_{\nu,a,\pm}(&t_n-t_{n-1};x_{n-1},x_n) \mathrm{d}m_{\nu,a,\pm}(x_n).
\end{align*}
In particular
\[
\mathbf{n}_{\nu,a,\pm}(e(t)\in \mathrm{d}x)=\pi_{\nu,a,\pm}(t;x)\mathrm{d}m_{\nu,a,\pm}(x),
\]
and it holds that
\[
\mathbf{n}_{\nu,a,\pm}(\zeta>t)=\int_0^\infty \mathbf{n}_{\nu,a,\pm}(e(t)\in \mathrm{d}x)=\int_0^\infty \pi_{\nu,a,\pm}(t;x) \, \mathrm{d}m_{\nu,a,\pm}(x).
\]
\end{theorem}
 
\begin{remark}
Note that $\mathbf{n}_{\nu,a,\pm}(\zeta>t)$ also has the expression
\[
\mathbf{n}_{\nu,a,\pm}(\zeta>t)=\int_t^\infty n_{\nu,a,\pm}(s) \, \mathrm{d}s=\int_0^\infty e^{-\lambda t} \sigma^*_{\nu,a,\pm}(\mathrm{d}t).
\]
This yields the identity (see \cite[Proposition 3]{SVY07})
\[
\int_0^t p_{\nu,a,\pm}(u;0,0)\, \mathrm{d}u  \int_{t-u}^\infty  n_{\nu,a,\pm}(v)\, \mathrm{d}v=1.
\]
\end{remark}

Recall that $H_x=\inf\{t> 0: w(t)= x\}$ is the first hitting time to $x$ and denote by $H^x=\sup\{t> 0:w(t)=x\}$  the last exit time from $x$. Let $(\cdot)^\vee$ denote the time reverse operator on $E^0=\{e\in E:e(0)=0\}$, that is, for $e\in E^0$,
\[
e^\vee(t)=e((\zeta-t)^+), t\ge 0.
\]
Let $\mathcal{E}^0$ denote the Borel $\sigma$-field of $E^0$ and let $\mathcal{E}^0_{(0,H_a)}$, $\mathcal{E}^0_{(H_a,H^a)}$ and $\mathcal{E}^0_{(H^a,\zeta)}$ denote the sub $\sigma$-fields with respect to the corresponding time intervals (see \cite{Ya06} for more precise definitions). Let $\theta_t(w)(\cdot)=w(t+\cdot)$ denote the left-shift operator. We have the following time reverse and first-entrance-last-exit decomposition of the excursion measure $\mathbf{n}_{\nu,a,\pm}$ from \cite{Ya06}.

\begin{theorem}\ 
\begin{itemize}
\item[(i)] For $\Gamma\in \mathcal{E}^0$ one has
\[
\mathbf{n}_{\nu,a,\pm}(\Gamma^\vee)=\mathbf{n}_{\nu,a,\pm}(\Gamma).
\]
\item[(ii)] For $x>0$ and $\Gamma_1 \in \mathcal{E}^0_{(0,H_x)}$, $\Gamma_2 \in \mathcal{E}^0_{(H_x,H^x)}$, $\Gamma_3 \in \mathcal{E}^0_{(H^x,\zeta)}$ one has
\[
\mathbf{n}_{\nu,a,\pm}(\Gamma_1\cap \Gamma_2 \cap \Gamma_3)=\frac{1}{x} \mathbf{P}^0_{h\text{-}\nu,a,\pm}(\Gamma_1)\mathbf{Q}^x_{\nu,a,\pm}(\theta_{H_x}(\Gamma_2)) \mathbf{P}^0_{h\text{-}\nu,a,\pm}(\Gamma_3^\vee).
\]
In particular
\begin{align*}
&\mathbf{n}_{\nu,a,\pm}(\{H_x\in \mathrm{d}t_1\}\cap \{H^x-H_x\}\in \mathrm{d}t_2\} \cap \{\zeta-H^x\}\in \mathrm{d}t_3)\\
=& \frac{1}{x}\mathbf{P}^0_{h\text{-}\nu,a,\pm}(H_x\in \mathrm{d}t_1)\mathbf{Q}^x_{\nu,a,\pm}(H^a\in \mathrm{d}t_2) \mathbf{P}^0_{h\text{-}\nu,a,\pm}(H_x\in \mathrm{d}t_3).
\end{align*}
Consequently, $\mathbf{n}_{\nu,a,\pm}(H_x\in \mathrm{d}t)=\frac{1}{x}\mathbf{P}^0_{h\text{-}\nu,a,\pm}(H_x\in \mathrm{d}t)$ and
\[
\mathbf{n}_{\nu,a,\pm}(e^{-\lambda H_x})=\frac{1}{x}\mathbf{P}^0_{h\text{-}\nu,a,\pm}(e^{-\lambda H_x})=\frac{1}{\psi_{\nu,a,\pm}(x,\lambda)}.
\]
\end{itemize}
\end{theorem}

\section{Probabilistic asymptotic behaviours of Liouville Brownian motion and Liouville Brownian excursions}\label{aLBM}

As an application of the spectral representation in Section \ref{srLBM} and \ref{srLBE}, we shall study the probabilistic asymptotic behaviours of LBM and LBE. Throughout this section let $\nu$ be a Borel measure on $\mathbb{R}$ satisfying (A1) as well as 
\begin{itemize}
\item[(A2)] Ergodicity: There exists a positive constant $Z$ such that for every $a\in\mathbb{R}$,
\begin{equation}\label{et}
\lim_{\eta\to \infty}\frac{\nu[a,a+\eta]}{\eta}=\lim_{\eta\to \infty}\frac{\nu[a-\eta,a]}{\eta}=Z.
\end{equation}

\item[(A3)] Multifractality: There exists an open interval $I_\nu$, a family $\{\nu_q: q\in I_\nu\}$ of Borel measures  on $\mathbb{R}$ and a family $\{\alpha(q):q\in I_\nu\}$ of positive reals such that for $q\in I_\nu$, for $\nu_q$-almost every $a\in \mathbb{R}$,
\begin{equation}\label{ma}
\lim_{r\to 0} \frac{1}{\log r} \log \nu(a-r,a+r)=\alpha(q).
\end{equation}
\end{itemize}

Many stationary multifractal random measures have these properties, for example the log-infinitely divisible cascade measures constructed in \cite{BM02,BM03}. In particular, an instance of the boundary Liouville measure constructed in Section \ref{lm} satisfies (A2) and (A3):
\begin{itemize}
\item[(1)] Since the boundary Liouville measure $\nu$ is a stationary positive measure on $\mathbb{R}$, that is $\nu(x+\cdot)$ has the same law as $\nu(\cdot)$ for any $x\in \mathbb{R}$ and $\nu(I)>0$ for any open interval $I$, by Birkhoff  ergodic theory there exists a positive random variable $Z$ with finite mean such that almost surely for every $a\in\mathbb{R}$,
\[
\lim_{\eta\to \infty}\frac{\nu[a,a+\eta]}{\eta}=\lim_{\eta\to \infty}\frac{\nu[a-\eta,a]}{\eta}=Z.
\]

\item[(2)] For $q\in (-\frac{\sqrt{2}}{\gamma},\frac{\sqrt{2}}{\gamma})$ let $\nu_q$ be the boundary Liouville measure with parameter $q\gamma$ defined via the same GFF $\mathfrak{h}$ as $\nu$. In particular $\nu_0$ is the Lebesgue measure on $\mathbb{R}$ and $\nu_1=\nu$. By the multifractal analysis of $\nu$ (see \cite{BM04II}, see also \cite[Theorem 4.1]{RV14a} for a direct proof  for positive $q$) we have that almost surely for $\nu_q$-almost every $a\in \mathbb{R}$,
\[
\lim_{r\to 0} \frac{1}{\log r} \log \nu(a-r,a+r)=1+(\frac{1}{2}-q)\frac{\gamma^2}{2}.
\]
\end{itemize}

As the first application we have the following theorem on the fractal dimensions of the level sets of LBM. For $a\in\mathbb{R}$ denote by $m_{\nu,a}=m_{\nu,a,+}+m_{\nu,a,-}$ and let
\[
V_{\nu,a}(r)=\int_0^r m_{\nu,a}(x) \,\mathrm{d}x,  \ r\ge 0.
\]
Let $h_{\nu,a}$ be the Krein's correspondence of $m_{\nu,a}$, which satisfies
\[
\frac{1}{h_{\nu,a}(\lambda)}=\frac{1}{h_{\nu,a,+}(\lambda)}+\frac{1}{h_{\nu,a,-}(\lambda)}.
\]
From \cite{KW82} we have for $\eta>0$
\begin{equation}\label{hv}
\frac{1}{4} h_{\nu,a}(1/\eta) \le (V_{\nu,a})^{-1}(\eta) \le 64 h_{\nu,a}(1/\eta).
\end{equation}

\begin{theorem}\label{level}
For $\nu_q$-almost every $a\in \mathbb{R}$, for $\mathbf{P}^a_\nu$-almost every $w\in W$,
\[
\dim_H \{t\ge 0:w(t)=a\}=\dim_P \{t\ge 0:w(t)=a\}=\frac{1}{1+\alpha(q)}.
\]
\end{theorem}

\begin{proof}

By \eqref{ma}, for $\nu_q$-almost every $a\in\mathbb{R}$ for every $\epsilon>0$ there exists $r_{a,\epsilon}>0$ such that 
\[
r^{\alpha(q)+\epsilon} \le \nu(a-r,a+r)\le r^{\alpha(q)-\epsilon}, \ \forall r\le r_{a,\epsilon}.
\]
This implies that
\begin{equation}\label{var}
\frac{1}{1+\alpha(q)+\epsilon} r^{1+\alpha(q)+\epsilon}\le V_{\nu,a}(r) \le \frac{1}{1+\alpha(q)-\epsilon} r^{1+\alpha(q)-\epsilon}, \ \forall r\le r_{a,\epsilon}.
\end{equation}
Equivalently for $\eta>0$ small enough
\[
(1+\alpha(q)-\epsilon)^{\frac{1}{1+\alpha(q)-\epsilon}} \eta^{\frac{1}{1+\alpha(q)-\epsilon}}\le (V_{\nu,a})^{-1}(\eta) \le(1+\alpha(q)+\epsilon)^{\frac{1}{1+\alpha(q)+\epsilon}} \eta^{\frac{1}{1+\alpha(q)+\epsilon}}.
\]
By \eqref{hv} we deduce that there exist constants $0<c_{a,\epsilon}, C_{a,\epsilon}<\infty$ such that for $\lambda$ large enough,
\begin{equation}\label{hal}
c_{a,\epsilon} \lambda^{-\frac{1}{1+\alpha(q)-\epsilon}}\le h_{\nu,a}(\lambda)\le C_{a,\epsilon} \lambda^{-\frac{1}{1+\alpha(q)+\epsilon}}.
\end{equation}

The local time $L_\nu(t,a)=L(\tau_\nu(t),a)$ at $a$ is carried by the level set $\{t\ge 0: w(t)=a\}$ for $\mathbf{P}^a_\nu$-almost every $w\in W$ and its the right-continuous inverse $\ell_{\nu,a}(t):=\inf\{s\ge 0 : L_\nu(s,a)>t\}$ is a L\'evy subordinator, whose L\'evy exponent is given by
\[
\mathbf{E}^a_{\nu}(e^{\lambda\ell_{\nu,a}(t)})=e^{-t/h_{\nu,a}(\lambda)}.
\]
By the general theory of fractal dimensions of images of L\'evy subordinator (see \cite[Chapter 5]{Ber99} for example), we have for $\mathbf{P}^a_\nu$-almost every $w\in W$,
\begin{align*}
\dim_H \{t\ge 0: w(t)=a\}=&\liminf_{\lambda\to \infty} \frac{-\log h_{\nu,a}(\lambda)}{\log \lambda};\\
\dim_P \{t\ge 0: w(t)=a\}=&\limsup_{\lambda\to \infty} \frac{-\log h_{\nu,a}(\lambda)}{\log \lambda}.
\end{align*}
By \eqref{hal} with $\epsilon \to 0$ we get for $\mathbf{P}^a_\nu$-almost every $w\in W$,
\[
\dim_H \{t\ge 0:w(t)=a\}=\dim_P \{t\ge 0:w(t)=a\}=\frac{1}{1+\alpha(q)}.
\]
\end{proof}

\begin{remark}
Theorem \ref{level} is linked to work \cite{Ja14} of Jackson on the Hausdorff dimension of the times that planar LBM spent on the thick points of the corresponding Gaussian free field. Theorem \ref{level} estimates the size of the times that one-dimensional LBM spent at $\nu_q$-almost every $a$, whereas \cite{Ja14} estimates the size of planar LBM spent in the support of $\mu_\gamma$. So, roughly speaking, Theorem \ref{level} can be considered as a fiber version of the result in \cite{Ja14} in dimension $1$. Since in dimension $2$ there does not exist the local time of BM/LBM at a given point, it seems difficult to derive an analogue of Theorem \ref{level} in dimension two.
\end{remark}

As the second application we shall estimate the asymptotic behaviours of the transition density $p_\nu(t;a,a)$ at a given point $a\in\mathbb{R}$. First note that $p_\nu(t;a,a)$ has the following spectral representation:
\[
p_{\nu}(t;a,a)=\int_0^\infty e^{-t\lambda} \, \sigma_{\nu,a}(\mathrm{d}\lambda),
\]
where $\sigma_{\nu,a}$ be the spectral measure of $m_{\nu,a}$, that is the unique non-negative Borel measure on $[0,\infty)$ with $\int_{(0,\infty)} \frac{\sigma_{\nu,a}(\mathrm{d}\xi)}{1+\xi}<\infty$ such that
\[
h_{\nu,a}(\lambda)=\int_0^\infty \frac{\sigma_{\nu,a}(\mathrm{d}\xi)}{\lambda+\xi}.
\]
This yields the following lemma of Tomisaki \cite{To77}.
\begin{lemma}\label{tol}
Let $\phi$ be a positive and non-increasing function on $(0,\delta)$ for some $\delta>0$. Then
\[
\int_{0}^\delta \phi(t)  p_{\nu}(t;a,a) \, \mathrm{d}t <\infty \Leftrightarrow \int_{0}^{(V_{\nu,a})^{-1}(\delta)} \phi(V_{\nu,a}(x)) \, \mathrm{d}x<\infty.
\]
\end{lemma}

We have the following result on the short term behaviour of $p_\nu(t;a,a)$.

\begin{theorem}\label{stb}
For $\nu_q$-almost every $a\in \mathbb{R}$,  for any $\beta>\frac{1}{1+\alpha(q)}$,
\[
\int_{0+} t^{-\beta} p_{\nu}(t;a,a) \, \mathrm{d}t=\infty,
\]
and for any $\beta<\frac{1}{1+\alpha(q)}$,
\[
\int_{0+} t^{-\beta} p_{\nu}(t;a,a) \, \mathrm{d}t<\infty.
\]
In particular
\begin{equation}\label{lst}
\liminf_{t\to 0} \frac{\log p_{\nu}(t;a,a)}{-\log t}\le 1-\frac{1}{1+\alpha(q)} \le \limsup_{t\to 0} \frac{\log p_{\nu}(t;a,a)}{-\log t}.
\end{equation}
\end{theorem}

\begin{proof}
As a direct consequence of Lemma \ref{tol} and \eqref{var} we have that for any $\beta>\frac{1}{1+\alpha(q)}$,
\[
\int_{0+} t^{-\beta} p_{\nu}(t;a,a) \, \mathrm{d}t=\infty,
\]
and for any $\beta<\frac{1}{1+\alpha(q)}$,
\[
\int_{0+} t^{-\beta} p_{\nu}(t;a,a) \, \mathrm{d}t<\infty,
\]
which implies \eqref{lst}.
\end{proof}

\begin{remark}
By using Tauberian theorem it can be shown that if $\nu(a-r,a+r)$ is a regular varying function of $r$ as $r\to 0$ then the inequalities in \eqref{lst} become an equality, see \cite{BK01} for example for the case when $\nu$ is a Bernoulli measure on $[0,1]$. However due to the multifractal nature of GMC measures, it is the case that for $\nu_q$ almost every $a$, $\nu(a-r,a+r)$ is not regular varying as $r\to 0$. So it is not clear whether the limit in \eqref{lst} exists.
\end{remark}

\begin{remark}
The short term behavior \eqref{lst} is quite different comparing to \cite[Corollary 2.1]{RV14}. The reason is that in \cite{RV14} the one-dimensional Liouville Brownian motion is defined as a linear diffusion with scale function $m_\nu$ and speed measure $\mathrm{d}x$, and the corresponding transition density $p_*(t;x,x)$ is defined with respect to $\nu(\mathrm{d}x)$ rather than $\mathrm{d}x$. Therefore by change of variables it is straightforward to verify that for every $x\in\mathbb{R}$,
\[
\lim_{t\to 0} \frac{-\log p_*(t;x,x)}{\log t}=\frac{1}{2}.
\]
\end{remark}

We also have the following long term behaviour of $p_{\nu}(t;a,a)$. 

\begin{theorem}\label{ltt}
For every $a\in \mathbb{R}$,
\[
\lim_{t\to \infty} \sqrt{2\pi t}p_{\nu}(t;a,a)=\frac{1}{\sqrt{Z}}.
\]
\end{theorem}
\begin{proof}
By the ergodicity \eqref{et}, for each $x\ge 0$ we have
\[
m_{\nu,a,\pm}^{(\eta)}(x):=\frac{1}{\eta Z}m_{\nu,a,\pm}(\eta x) \to x \text{ as } \eta\to \infty.
\]
By change of variables it is easy to see that for constants $\eta,\xi>0$ one has the following relation of the Krein's correspondence:
\[
\frac{\xi}{\eta} m(\frac{x}{\eta}) \leftrightarrow \xi h(\eta\lambda).
\]
Since Krein's correspondence is a homeomorphism, when $m_{\nu,a,\pm}^{(\eta)}(x) \to x$ as $\eta\to \infty$ for each $x\ge 0$, the corresponding generalized linear diffusion process converges in law to one-dimensional Brownian motion. This implies that
\[
\lim_{\eta\to \infty}\eta Z p_{\nu}(t\eta^2Z;a,a) = p_{\mathrm{BM}}(t;a,a)=\frac{1}{\sqrt{2\pi t}}
\]
In other words,
\[
\lim_{t\to \infty} \sqrt{2\pi t}p_{\nu}(t;a,a)=\frac{1}{\sqrt{Z}}.
\]
\end{proof}

In the third application we shall study the first hitting/exit time of LBM. For $a\in \mathbb{R}$ recall that $H_a=\inf\{t>0: w(t)=a\}$ is the first hitting time at $a$. We have
\begin{theorem}\label{lth}
For $a\neq 0$ we have
\[
\lim_{t\to\infty} \sqrt{2\pi t}\mathbf{P}^0_\nu(H_a\ge t) = |a|Z.
\]
\end{theorem}
\begin{proof}
Since
\[
\lim_{\eta\to \infty}\frac{\nu[0,\eta]}{\eta}=\lim_{\eta\to \infty}\frac{\nu[-\eta,0]}{\eta}=Z,
\]
the result is a direct application of \cite[Theorem 4]{Ya90} with $\alpha=\frac{1}{2}$ and $K(x)=Z$.
\end{proof}

\begin{remark}
Theorem \ref{ltt} and \ref{lth} suggest that in long term one-dimensional LBM behaves exactly like one-dimensional Brownian motion.
\end{remark}

For $a<0<b$ let
\[
H_{a,b}=\inf\{t>0:w(t)\not\in(a,b)\}=H_a\wedge H_b
\]
denote the first exit time from $(a,b)$. Define
\[
C_{\nu,a,b}=\frac{1}{b-a}\int_a^b(b-x)(x-a) \, \nu(\mathrm{d}x).
\]
Then we have
\begin{theorem}\label{exit}
If $\lambda < 1/C_{\nu,a,b}$, then
\[
\mathbf{E}^0_\nu(e^{\lambda H_{a,b}})<\infty.
\]
\end{theorem}
\begin{proof}
This can be easily deduced by using the Kac formula: for $n\ge 1$,
\[
\mathbf{E}^0_\nu(H_{a,b}^n)=n\int_a^b \frac{(b-x\vee 0)(x\wedge 0-a)}{b-a} \mathbf{E}^x_\nu(H_{a,b}^{n-1}) \, \nu(\mathrm{d}x)=n! G_\nu^n1(0),
\]
where $G_\nu$ is the Green operator
\[
G_\nu f(x)=\int_a^b \frac{(b-y\vee x)(y\wedge x-a)}{b-a} f(y) \, \nu(\mathrm{d}y).
\]
See \cite[Lemma 1.3]{LLS11} for example.
\end{proof}

Theorem \ref{exit} indicates that $A=-\frac{\mathrm{d}}{\mathrm{d}\nu}\frac{\mathrm{d}}{\mathrm{d}x}$, as a self-adjoint, non-negative definite operator on the Hilbert space $L^2((a,b),\nu)$, has a spectra gap. Indeed let $\lambda_{\nu,a,b}$ denote the smallest eigenvalue of $A$ on $L^2((a,b),\nu)$, and denote by
\[
\widetilde{C}_{\nu,a,b}=\sup_{x\in(a,0]}(x-a)\nu((x,0]) \vee \sup_{x\in[0,b)} (b-x) \nu([0,x)).
\]
Then we have the following result of Katoni \cite[Theorem 3, Appendix I]{KW82} as an extension of the theorem of Kac and Krein \cite{KK58}.
\begin{theorem}
\[
\widetilde{C}_{\nu,a,b} \le \lambda_{\nu,a,b}^{-1}\le 4 \widetilde{C}_{\nu,a,b}.
\]
\end{theorem}

In the last application we study the asymptotic behaviours of the lifetime of LBE. Recall that $n_{\nu,a,\pm}(t)$ is the density of the inverse local time $\ell_{\nu,a,\pm}$, which is also the density of the lifetime $\zeta$ under the excursion measure $\mathbf{n}_{\nu,a,\pm}$, that is
\[
\mathbf{n}_{\nu,a,\pm}(\zeta\in \mathrm{d}t)=n_{\nu,a,\pm}(t) \, \mathrm{d}t.
\]

First we present the asymptotic behaviour of $ \mathbf{n}_{\nu,a,\pm}(\zeta>t) $ as $t\to \infty$.

\begin{theorem}\label{ellong}
For $a\in\mathbb{R}$ we have
\[
\lim_{t\to \infty} 2\sqrt{2\pi t^3} n_{\nu,a,\pm}(t)=\frac{1}{\sqrt{Z}}.
\]
Consequently
\[
\lim_{t\to \infty} \sqrt{2\pi t} \mathbf{n}_{\nu,a,\pm}(\zeta>t) = \frac{1}{\sqrt{Z}}.
\]
\end{theorem}

\begin{proof}
Similar as in the proof of Theorem \ref{ltt}, for each $x\ge 0$,
\[
m_{\nu,a,\pm}^{(\eta)}(x):=\frac{1}{\eta Z}m_{\nu,a,\pm}(\eta x) \to x \text{ as } \eta\to \infty.
\]
This implies that the generalized diffusion process on $\mathbb{R}_+$ with natural scale function and speed measure $\mathrm{d}m_{\nu,a,\pm}^{(\eta)}(x)$, and with $0$ as an instantaneously reflecting boundary converges in law to the one-dimensional reflected Brownian motion as $\eta\to \infty$. Therefore the corresponding local time
\[
\ell_{\nu,a,\pm}^{(\eta)}(t):=\frac{1}{\eta^2 Z}\ell_{\nu,a,\pm}(t\eta)
\]
converges in law to the $\frac{1}{2}$-stable L\'evy subordinator as $\eta\to \infty$. Since
\begin{align*}
\mathbf{E}^0_{\nu,a,\pm}\left(e^{-\lambda \frac{1}{\eta^2 Z}\ell_{\nu,a,\pm}(t\eta)}\right)=& \exp\left(-t\eta \int_0^\infty (1-e^{-s \lambda \frac{1}{\eta^2 Z} }) n_{\nu,a,\pm}(s) \, \mathrm{d}s\right)\\
=&\exp\left(-t \int_0^\infty (1-e^{-u \lambda  }) \eta^3 Z n_{\nu,a,\pm}(u \eta^2 Z) \, \mathrm{d}u\right),
\end{align*}
we get that
\[
\eta^3 Z \cdot n_{\nu,a,\pm}(u \eta^2 Z) \to \frac{1}{2\sqrt{2\pi u^3}} \text{ as } \eta\to\infty.
\]
In other words,
\[
\lim_{t\to \infty} 2\sqrt{2\pi t^3} \cdot n_{\nu,a,\pm}(t)=\frac{1}{\sqrt{Z}}.
\]
Consequently
\[
\lim_{t\to \infty} \sqrt{2\pi t} \mathbf{n}_{\nu,a,\pm}(\zeta>t) = \frac{1}{\sqrt{Z}}.
\]
\end{proof}

Now we present the asymptotic behaviour of $ \mathbf{n}_{\nu,a,\pm}(\zeta>t) $ as $t\to 0$.

\begin{theorem}\label{elshort}
For $\nu_q$-almost every $a\in\mathbb{R}$, for any $\beta>\frac{1}{1+\alpha(q)}$,
\[
\int_{0+} t^{\beta-1} \mathbf{n}_{\nu,a,\pm}(\zeta>t)  \, \mathrm{d}t<\infty,
\]
and for any $\beta<\frac{1}{1+\alpha(q)}$,
\[
\int_{0+} t^{\beta-1}  \mathbf{n}_{\nu,a,\pm}(\zeta>t)  \, \mathrm{d}t=\infty.
\]
In particular,
\begin{equation}\label{nst}
\liminf_{t\to 0+} \frac{\log  \mathbf{n}_{\nu,a,\pm}(\zeta>t) }{-\log t}\le \frac{1}{1+\alpha(q)} \le \limsup_{t\to 0+} \frac{\log  \mathbf{n}_{\nu,a,\pm}(\zeta>t) }{-\log t}.
\end{equation}
\end{theorem}

\begin{proof}
We have, for $\lambda>0$,
\[
\frac{1}{\lambda h_{\nu,a,\pm}(\lambda)}=\int_0^\infty e^{-\lambda t}\mathbf{n}_{\nu,a,\pm}(\zeta>t) \,\mathrm{d}t.
\]
Therefore for any $\delta>0$ and $\beta>0$
\begin{align*}
\int^{\infty}_\delta\frac{1}{\lambda^{\beta+1} h_{\nu,a,\pm}(\lambda)} \, \mathrm{d}\lambda=&\int^{\infty}_\delta \lambda^{-\beta}\int_{0}^\infty e^{-\lambda t} \mathbf{n}_{\nu,a,\pm}(\zeta>t) \, \mathrm{d}t \mathrm{d}\lambda\\
=&\int^{\infty}_{\delta t} \lambda^{-\beta}e^{-\lambda} \,\mathrm{d}\lambda\int_{0}^\infty t^{\beta-1} \mathbf{n}_{\nu,a,\pm}(\zeta>t) \, \mathrm{d}t.
\end{align*}
This implies that
\[
\int^{\infty-}\frac{1}{\lambda^{\beta+1} h_{\nu,a,\pm}(\lambda)} \, \mathrm{d}\lambda <\infty \ \Leftrightarrow \int_{0+} t^{\beta-1} \mathbf{n}_{\nu,a,\pm}(\zeta>t) \, \mathrm{d}t<\infty.
\]
By \eqref{hal} this yields that for any $\beta>\frac{1}{1+\alpha(q)}$,
\[
\int_{0+} t^{\beta-1} \mathbf{n}_{\nu,a,\pm}(\zeta>t)  \, \mathrm{d}t<\infty,
\]
and for any $\beta<\frac{1}{1+\alpha(q)}$,
\[
\int_{0+} t^{\beta-1}  \mathbf{n}_{\nu,a,\pm}(\zeta>t)  \, \mathrm{d}t=\infty,
\]
which implies \eqref{nst}.
\end{proof}

%

\bibliographystyle{abbrv}

\end{document}